\documentclass[12pt,a4paper,reqno]{amsart}

\usepackage[sc, noBBpl]{mathpazo}	
\usepackage[bbgreekl]{mathbbol}
\usepackage{amssymb, amsmath, amsthm, amsfonts}

\DeclareSymbolFontAlphabet{\mathbb}{AMSb}
\DeclareSymbolFontAlphabet{\mathbbl}{bbold}

\usepackage[T1]{fontenc}
\usepackage[utf8]{inputenc}
\usepackage{graphicx}
\usepackage[linkbordercolor={0.8 0.9 0.9}]{hyperref}
\usepackage{comment}
\usepackage{latexsym,epsfig}
\usepackage[all]{xy}
\usepackage[usenames,dvipsnames]{color}

\newcommand{\st}{\;|\;}

\newcommand{\textfrac}[2]{{\textstyle\frac{#1}{#2}}}

\DeclareMathOperator{\Exp}{Exp}

\newcommand{\lie}{\mathfrak}

\newcommand{\UU}{\mathbb{U}}

\newcommand{\RR}{\mathbb{R}}

\newcommand{\TT}{\mathbb{T}}

\newcommand{\XX}{\mathbb{X}}
\newcommand{\YY}{\mathbb{Y}}

\newcommand{\bbnabla}{\boldsymbol{\nabla}}

\newcommand{\XH}{\mathfrak{X}_H}

\newcommand{\tHM}[1][]{\lie{t}^{#1}_HM}

\newcommand{\THM}{T_HM}
\newcommand{\ttHM}{\mathbbl{t}_HM}
\newcommand{\ttM}{\mathbbl{t}M}
\newcommand{\TTHM}{\mathbb{T}_HM}
\newcommand{\aHG}[1][]{\lie{a}_H^{#1}G}
\newcommand{\AG}{\mathrm{A}G}
\newcommand{\AHG}{\mathrm{A}_H G}
\newcommand{\AAHG}{\mathbb{A}_HG}
\newcommand{\aaHG}{\mathbbl{a}_HG}

\newcommand{\EExp}{\mathbb{E}\mathrm{xp}}
\newcommand{\ev}{\mathrm{ev}}
\newcommand{\id}{\mathrm{id}}

\newcommand{\slot}{\,\cdot\,}
\newcommand{\anchor}{\rho}

\newcommand{\defeq}{:=}

\newtheorem{theorem}{Theorem}
\newtheorem{lemma}[theorem]{Lemma}
\newtheorem{proposition}[theorem]{Proposition}

\theoremstyle{definition}
\newtheorem{definition}[theorem]{Definition}
\newtheorem{example}[theorem]{Example}

\theoremstyle{remark} 
\newtheorem{remark}[theorem]{Remark}

\title{On the tangent groupoid of a filtered manifold}
\author{Erik van Erp}
\address{Department of Mathematics, Dartmouth College, Hanover, NH 03755, USA}
\email{erikvanerp@dartmouth.edu}

\author{Robert Yuncken}
\thanks{R.~Yuncken was supported by the project SINGSTAR of the Agence Nationale de la Recherche, ANR-14-CE25-0012-01.}
\address{Université Clermont Auvergne, Université Blaise Pascal, BP 10448, F-63000 Clermont-Ferrand, France}
\email{yuncken@math.univ-bpclermont.fr}

\date{}         
\subjclass[2010]{Primary: 58H05; Secondary: 22A22, 58J40, 35S05, 47G30}
\keywords{Tangent groupoid; Lie groupoids; filtered manifolds}

\begin{document}

\begin{abstract}
 We give an intrinsic (coordinate-free) construction of the tangent groupoid of a filtered manifold.  This is an analogue of Connes' tangent groupoid which is pertinent for the analysis of certain subelliptic differential operators.  It is a deformation of the pair groupoid to a  bundle of nilpotent groups.  We also describe the analogous construction in the context of the adiabatic groupoids of a filtered Lie groupoid.
 \end{abstract}

\maketitle


\section{Introduction}

Connes' tangent groupoid 
\[  \TT M = (M\times M) \times \RR^\times \; \sqcup \; TM \times \{0\} \]
provides a powerful conceptual framework for the study of pseudodifferential operators. 
Connes famously used it in a proof of the Atiyah-Singer Index Theorem \cite{Connes:NCG}.  

Recall that pseudodifferential operators on a manifold $M$ are defined by Schwartz kernels on $M\times M$ with singularities on the diagonal belonging to some prescribed class.  
The power of the tangent groupoid lies in the smooth gluing of the fibres $M\times M$ (which carries the Schwartz kernel of a pseudodifferential operator) to the tangent space $TM$ (which carries the principal symbol).

By now we have tangent groupoids associated to many different classes of pseudodifferential operators.  One example is the Heisenberg calculus which is relevant to contact and CR geometries.  In this case, the tangent bundle $TM$ is replaced by a bundle of Heisenberg groups $\THM$ called the \emph{osculating groupoid}.  The Heisenberg calculus was developed by Folland-Stein \cite{FolSte:Estimates} and Taylor \cite{Taylor:Microlocal} in the 1970s.
An adaptation of the  tangent groupoid to the Heisenberg calculus was developed in \cite{VanErp:Thesis, Ponge:groupoid}, which led to the index theorem for the Heisenberg calculus  \cite{VanErp:I, VanErp:II}.

Taking this further, one would expect analogous constructions in the context of filtered manifolds.    In heuristic terms, a filtered manifold (called a \emph{Carnot manifold} by Choi-Ponge \cite{ChoPon}) is a manifold for which some  vector fields are viewed as differential operators of degree greater than one.

We have shown in \cite{VanYun:PsiDOs} that a pseudodifferential calculus for filtered manifolds, à la Melin \cite{Melin:Lie_filtrations}, can be produced from a tangent groupoid for filtered manifolds.  The relevant tangent groupoid appears in Choi-Ponge \cite{ChoPon}, with a construction based on preferred local coordinate systems.  For our work we needed an instrinsic (coordinate-free) construction of the same groupoid.  The goal of this paper is to describe this construction.

The results in this paper were originally contained in the preprint \cite{VanYun:PsiDOs}.  Following a referee's suggestion, that paper has been split into two conceptually distinct pieces: the present part on the tangent groupoid, and the second part on pseudodifferential calculi.

\subsection{Statement of results}

In \cite[\S II.5]{Connes:NCG} Connes constructs the smooth structure of $\TT M$ near  $t=0$  by means of a  map $TM\times\RR \to \TT M$, 
 \begin{equation*}
   (x,v,t) \mapsto 
  \begin{cases}
    (\exp^{\nabla}(tv),x,t), & t\neq 0 \\
   (x,v,0), & t=0.
  \end{cases}
 \end{equation*}
Here $\exp^\nabla:TM\to M$ is the exponential map associated to  a connection $\nabla$ on $TM$.
Adapting this construction to filtered manifolds is  more challenging than it might seem at first.

\begin{definition}
\label{def:filtered_manifold}
A {\em filtered manifold} is a smooth manifold $M$ equipped with a filtration of the tangent bundle $TM$ by vector bundles $M\times\{0\} = H^0 \leq H^1 \leq \cdots \leq H^N = TM$ such that $\Gamma^\infty(H^\bullet)$ is a Lie algebra filtration---\emph{i.e.},
\[[ \Gamma^\infty(H^i), \Gamma^\infty(H^j) ] \subseteq \Gamma^\infty(H^{i+j}).\]
Here we are using the convention that $H^i = TM$ for $i>N$.  
Elements of $\Gamma^\infty(H^i)$ may be referred to as \emph{vector fields of order (less than or equal to) $i$}.
\end{definition}

Given such a filtration on $TM$, the associated graded bundle $\tHM \defeq \bigoplus_i H^i/H^{i-1}$ inherits the structure of a bundle of nilpotent Lie algebras over $M$---see Definition \ref{def:osculating_algebroid}.  The associated bundle of simply connected nilpotent Lie groups $\THM$ is called the \emph{bundle of osculating groups} or the \emph{osculating Lie groupoid}.  

\begin{theorem}
\label{thm:intro}
 Let $M$ be a filtered manifold.  Consider the family of Lie groupoids $(G_t)_{t\in\RR}$ with
 \[
  G_t \defeq \begin{cases}
    M\times M, & \text{if }t\neq 0 \\
    \THM, &\text{if } t=0.
  \end{cases}
 \]
 There is a canonical smooth structure on the union
 \begin{align*}
  \TTHM &= \bigsqcup_{t\in\RR}  G_t \nonumber 
    = (M\times M) \times \RR^\times \; \sqcup \; \THM \times \{0\}, \label{eq:TTHM}
 \end{align*}
 which is compatible with the smooth structures on the two components and makes $\TTHM$ into a Lie groupoid.
\end{theorem}

An explicit smooth chart for the groupoid $\THM$ in a neigborhood of the $t=0$ fiber is given by the map $TM\times\RR \to \TTHM$ with
 \begin{equation*}
   (x,v,t) \mapsto 
  \begin{cases}
    (\exp^{\nabla}(\delta_t(v)),x,t), & t\neq 0 \\
   (x,v,0), & t=0.
  \end{cases}
 \end{equation*}
As in Connes' construction, $\exp^\nabla:TM\to M$ is the exponential map associated to  a connection $\nabla$ on $TM$,
while $tv$ is quite naturally replaced by $\delta_t(v)$, where $\delta_t$ are dilations of $TM$ compatible with the filtration (see Section  \ref{sec:dilations}).
What is not easy to see is that for the smooth structure of $\TTHM$ to be well-defined,
 the connection $\nabla$ must be compatible with the dilations, in the sense that
\[ \delta_t\circ \nabla\circ \delta_t^{-1} = \nabla \qquad t\ne 0\]
The  construction of $\TTHM$ in this paper procedes in natural stages.
The technical condition on the connection $\nabla$ will appear naturally at the end of the process.

To prove  Theorem \ref{thm:intro} we follow a classic strategy: just as it is easier to construct a Lie algebra than a Lie group, so it is easier to define a Lie algebroid than a Lie groupoid.  We therefore begin by constructing the Lie algebroid $\ttHM$ of $\TTHM$.  This is  a matter of gluing a filtered bundle to its associated graded bundle---a process which is most efficiently carried out by working with their modules of sections.  We  then appeal to well-known results on integrating Lie algeboids to produce the tangent groupoid $\TTHM$.

Finally, in Section \ref{sec:filtered_groupoids} we describe a generalization of our construction to filtered Lie groupoids $G$, obtaining an analogue of the adiabatic groupoid in the presence of a filtration on the Lie algebroid of $G$.

\subsection{Acknowledgements}
It is a pleasure to thank Claire Debord, Nigel Higson, Jean-Marie Lescure and Georges Skandalis for their input.

\begin{remark}
 Since the appearance of this article as a preprint, another approach to the tangent groupoid of a filtered manifold has appeared in \cite{HigSad} using methods closer to algebraic geometry.
\end{remark}


\section{Lie algebroids}
\label{sec:Lie_algebroids}

For the basics of Lie groupoids and Lie algebroids we refer the reader to \cite{MoeMrc}.  Throughout, we use the following notation: $G$ denotes a Lie groupoid; $M=G^{(0)}$ is the space of units; $r,s:G \to G^{(0)}$ are the range and source maps, respectively; $ \AG = \ker ds|_{G^{(0)}}$ is the Lie algebroid of $G$; and $\anchor = dr: \AG \to TG^{(0)}$ is the anchor.

Abstractly, a \emph{Lie algebroid} over a smooth manifold $M$ is a vector bundle $\lie{g} \to M$ equipped with two compatible structures:
\begin{itemize}
 \item A Lie bracket on smooth sections $[\slot,\slot]:\Gamma^\infty(\lie{g}) \times \Gamma^\infty(\lie{g}) \to \Gamma^\infty(\lie{g})$,
 \item A vector bundle map $\anchor:\lie{g}\to TM$, called the \emph{anchor},
\end{itemize}
such that
\begin{enumerate}
 \item the induced map on sections $\anchor:\Gamma^\infty(\lie{g}) \to \Gamma^\infty(TM)$ is a Lie algebra homomorphism,
 \item for any $X,Y \in \Gamma^\infty(\lie{g})$, $f\in C^\infty(M)$,
 \[
  [X,fY] = f[X,Y] + (\anchor(X)f) Y.
 \]
\end{enumerate}
In nice situations (such as we will have in this paper) one has an analogue of Lie's Third Theorem, which associates to a Lie algebroid $\lie{g}$ a Lie groupoid $G$ with $\AG = \lie{g}$.

With the exception of the final Section \ref{sec:filtered_groupoids}, we will be working  with Lie groupoids and Lie algebroids that are built out of the following two simple examples.

\begin{example}
\label{ex:pair_algebroid}
 The \emph{tangent bundle} $TM \to M$ is a Lie algebroid over $M$, with the usual Lie bracket of vector fields and the identity map as anchor.  This is the Lie algebroid of the pair groupoid $M\times M$.
\end{example}

\begin{example}
\label{ex:bundle_of_Lie_algebras}
 If the anchor of a Lie algebroid $\lie{g}$ is the zero map, then $\lie{g}$ is a \emph{smooth bundle of Lie algebras}.  This is because the Lie bracket $[\slot,\slot]$ on the space of sections $\Gamma^\infty(\lie{g})$ is a $C^\infty(M)$-linear map, and so restricts to a well-defined Lie bracket on each fibre $\lie{g}_x$ ($x\in M$).  
\end{example}

As a particular case of the second example, the tangent bundle $TM$ has an alternative Lie algebroid structure as a bundle of abelian Lie algebras.  In this case, we equip $\Gamma^\infty(TM)$ with the zero Lie bracket and the zero anchor.  

The Lie algebroid $\ttM$ of  Connes' original tangent groupoid is obtained by interpolating smoothly between the two different Lie algebroid structures on $TM$ in the above examples.  This is easily achieved: we simply multiply the Lie algebroid structures on $TM$ by $t\in\RR$ to obtain a smooth family of Lie algebroids indexed by $t$.  Exponentiating this yields the tangent groupoid\footnote{We are ignoring certain technical issues with the exponential here, including simple-connectivity of the fibres, which will be properly addressed in Section \ref{sec:tangent_groupoid}.}.  

The analogue of this construction for filtered manifolds follows the same basic idea, except that (roughly speaking) the rescaling will be by $t^i$ on vector fields of degree $i$.  The goal of the rest of the paper is to make this precise.


\section{The osculating groupoid of a filtered manifold}
\label{sec:osculating}

Let $M$ be a filtered manifold (see Definition \ref{def:filtered_manifold}).  As described in the introduction, the pertinent analogue of the tangent bundle $TM$ is the bundle of osculating nilpotent groups.  
To construct this osculating groupoid we begin with a Lie algebroid, which we will then exponentiate.

The key point in constructing the osculating Lie algebroid is contained in the following simple calculation.  Let $M$ be a filtered manifold, and let $X\in\Gamma^\infty(H^i)$, $Y\in\Gamma^\infty(H^j)$ be vector fields of order $i$ and $j$, respectively.  For any $f,g \in C^\infty(M)$, we have
\begin{align*}
  [fX, gY] &= fg[X,Y] + f(\anchor(X)g)Y - g(\anchor(Y)f)X \\
  &= fg[X,Y] \qquad\mod{\Gamma^\infty(H^{i+j-1})}.
\end{align*}
Thus, although the Lie bracket on vector fields is not $C^\infty(M)$-linear, it does induce a $C^\infty(M)$-linear bracket on sections of the associated graded bundle.  Therefore, the associated graded bundle is a smooth bundle of Lie algebras.  Here are the details.

\begin{definition}
\label{def:osculating_algebroid}
 Let $M$ be a filtered manifold.  The \emph{osculating Lie algebroid} $\tHM$ is the associated graded bundle of $TM$:
 \[
  \tHM  := \bigoplus_{i=1}^N H^{i} / H^{i-1}.
 \]
 The Lie bracket on $\Gamma^\infty(\tHM)$ is that induced from the Lie bracket of vector fields $[\slot,\slot]:\Gamma^\infty(H^i) \times \Gamma^\infty(H^j) \to \Gamma^\infty(H^{i+j})$, and the anchor is the zero map.

 We write $\tHM[i] \defeq H^i/H^{i-1} \subset \tHM$ for the degree $i$ subbundle of $\tHM$, and $\sigma_i : H^i \to \tHM[i]$ for the grading maps.
\end{definition}

Since the anchor is zero, we are in the situation of Example \ref{ex:bundle_of_Lie_algebras}---that is, $\tHM$ is a bundle of graded nilpotent Lie algebras over $M$.

\begin{definition}
  The \emph{osculating groupoid} $\THM$ of a filtered manifold $M$ is the bundle of connected, simply connected nilpotent Lie groups which integrates\footnote{For basic results on integration of nilpotent Lie algebras and their morphisms, we refer the reader to \cite{Knapp:Lie_groups}, particularly \S{}I.16 and Appendix B.  Note that all Lie algebras in this article are finite dimensional.} the Lie algebroid $\tHM$.
\end{definition}

Explicitly, $\THM$ equals $\tHM$ as a smooth fibre bundle, and each fibre is equipped with the group law given by the Baker-Campbell-Hausdorff formula.  Note that in every fibre the Baker-Campbell-Hausdorff formula has finite length which is bounded by the depth of the filtration.  The groupoid multiplication is therefore smooth.

\section{Dilations}
\label{sec:dilations}

Any graded vector bundle admits a canonical one-parameter family $(\delta_\lambda)_{\lambda\in\RR}$ of bundle endomorphisms (automorphisms for $\lambda\neq0$) called the \emph{dilations}, where $\delta_\lambda$ acts on the degree $i$ subspace by multiplication by $\lambda^i$.  These generalize the homotheties of a vector bundle with trivial grading.  

We shall be interested in the osculating Lie algebroid  $\tHM$, where the dilations are defined by
\[
 \delta_\lambda : \xi  \mapsto \lambda^i \xi, \qquad \text{for all }\xi \in \tHM[i] = H^i/H^{i-1}.
\]
In this case, the dilations are Lie algebroid homomorphisms (isomorphisms if $\lambda\ne 0$). They integrate to Lie groupoid homomorphisms of the osculating groupoid $\THM$.

For more details on these dilations, see for instance \cite{HelNou, Melin:Lie_filtrations, ChrGelGloPol}.

\section{The tangent Lie algebroid I: Module of sections}

Algebraically, the tangent groupoid of a filtered manifold will be a disjoint union
\[
  \TTHM = (M\times M)\times\RR^\times \;\sqcup\; \THM \times \{0\},
\]
seen as a family of groupoids parametrized by $t\in\RR = \RR^\times \sqcup \{0\}$.  
The difficulty is to define the smooth manifold structure of this groupoid. 
This is the key point of the construction, and we will approach it in stages.

In what follows, $TM \times \RR$ is viewed as a vector bundle over $M\times \RR$ with the obvious projection---\emph{i.e.}, it is a constant family of bundles $TM\times\{t\}$ indexed by $t\in\RR$.  As a general point of notation, if $\XX$ is a section of some bundle over $M\times\RR$, we will write $\XX|_t$ for its restriction over $M\times\{t\}$.  

\begin{definition}
\label{def:module_of_sections}
  Let $\XH$ denote the $C^\infty(M\times\RR)$-submodule of $\Gamma^\infty(TM \times \RR)$ consisting of sections which vanish at $t=0$ to various orders as follows:
\begin{equation}
 \label{eq:module_of_sections}
\XH \defeq \{ \XX\in \Gamma^\infty(TM\times\RR) \st 
    \partial_t^k \XX|_{t=0}\in \Gamma^\infty(H^k) \text{ for all $k\geq 0$}\}
\end{equation}
where $\partial_t = \frac{\partial}{\partial t}$.
\end{definition}

In Section \ref{sec:II}, we will show that $\XH$ is in fact the space of smooth sections of a vector bundle $\ttHM$ over $M\times\RR$.
It turns out to be easier to define the module of sections $\XH=\Gamma^\infty(\ttHM)$ than it is to define the vector bundle $\ttHM$ itself. 
 In Section \ref{sec:III} we will define a Lie algebroid structure on $\ttHM$.  In Section \ref{sec:tangent_groupoid} we integrate this to obtain the smooth structure on $\TTHM$.

\section{The tangent Lie algebroid II: Vector bundle}
\label{sec:II}

As a vector bundle, the osculating Lie algebroid $\tHM$ of a filtered manifold is isomorphic, but not canonically isomorphic  to the tangent bundle $TM$.  To {show that $\XH$ is the module of sections of a vector bundle, we will need to transfer the dilations $\delta_\lambda$ from $\tHM$ to $TM$.  For this, we need a splitting.

\begin{definition}
 \label{def:splitting}
 A {\em splitting} of $\tHM$ will mean an isomorphism of vector bundles $\psi:\tHM \to TM$ such that for each $i$, the restriction of $\psi$ to $\tHM[i]$ is right inverse to the grading map $\sigma_i:H^i \to \tHM[i]$.
\end{definition}

\begin{proposition}
 \label{prop:bundle}
 Let $M$ be a filtered manifold and let $\XH$ be the $C^\infty(M\times\RR)$-module of Definition \ref{def:module_of_sections}.  For any splitting $\psi$, the map 
 \begin{align*}
  {\Phi^\psi}: \Gamma^\infty(\tHM\times\RR)& \to \XH && \subset \Gamma^\infty(TM\times\RR)\\
  ({\Phi^\psi}\XX)(x,t) &\defeq \psi(\delta_t(\XX(x,t)))
 \end{align*}
 is an isomorphism of $C^\infty(M\times\RR)$-modules.  

 Moreover, given any other splitting $\varphi$, the composition $(\Phi^\varphi)^{-1} \circ \Phi^\psi : \Gamma^\infty(\tHM\times\RR) \to \Gamma^\infty(\tHM\times\RR)$ is given by a smooth bundle automorphism of $\tHM \times \RR$ that is the identity at $t=0$.
\end{proposition}

\begin{proof}
Let $\YY\in\Gamma^\infty(\tHM\times\RR)$.  We decompose it into graded components as $\YY = \sum_{i=1}^N \YY_i$ with $\YY_i\in\Gamma^\infty(\tHM[i]\times\RR)$.   Then
 \[
  \partial_t^k (\Phi^\psi \YY)|_{t=0} 
   = \sum_{i=1}^N \partial_t^k (t^i \psi(\YY_i)) |_{t=0}.
 \] 
The summands on the right are zero for all $i>k$, so the sum belongs to $H^k$.  This proves that the image of $\Phi^\psi$ lies in the module $\XH$.  

The map $\Phi^\psi$ is obviously $C^\infty(M\times\RR)$-linear, and since $\psi\circ \delta_t$ is invertible for all $t\neq 0$, ${\Phi^\psi}$ is injective.  
We need to show it is surjective.

Let $\XX\in \XH \subset \Gamma^\infty(TM\times\RR)$.  Since the splitting is an isomorphism, we can write $\XX = \sum_{i=1}^N \psi(\XX_i)$ for some $\XX_i\in \Gamma^\infty(\tHM[i]\times\RR)$.  
By the definition of $\XH$, $\partial_t^k \XX|_{t=0} \in \Gamma^\infty(H^k)$ for all $k$, so we must have $\partial_t^k \XX_i|_{t=0} = 0$ for all $k<i$,
and therefore $\XX_i(x,t) = t^i\YY_i(x,t)$ for some $\YY_i\in\Gamma^\infty(\tHM[i] \times \RR)$.  
Then $\XX=\Phi^\psi (\YY)$, where $\YY =  \sum_{i=1}^N \YY_i$.  This proves surjectivity.

Now let $\varphi$ be a second splitting.  At any $(x,t)\in M\times\RR$ with $t\neq 0$ we have
 \[
   (\Phi^\varphi)^{-1}\circ \Phi^\psi (\YY )(x,t) = \delta_t^{-1}\circ\varphi^{-1}\circ\psi\circ\delta_t (\YY(x,t)),
 \]
 which is obviously induced by the bundle automorphism $\delta_t^{-1}\circ\varphi^{-1}\circ\psi\circ\delta_t$ on $\tHM\times\RR^\times$.  We claim that it extends to a smooth bundle automorphism also at $t=0$.
 
 Since $\psi$ and $\varphi$ are both splittings, $\varphi^{-1}\circ\psi$ preserves the top degree part of any vector in $\tHM$.  That is, if we fix a degree $k$, then $\varphi^{-1}\circ\psi$ maps $\tHM[k]$ into $\bigoplus_{j\leq k} \tHM[j]$, and it decomposes as
 \[
  \varphi^{-1}\circ\psi|_{\tHM[k]} = \id + \epsilon_{k-1} + \cdots + \epsilon_1
 \]
 for some linear maps $\epsilon_j:\tHM[k] \to \tHM[j]$ with $j<k$.  Therefore,
\[
 (\delta_t^{-1} \circ \varphi^{-1} \circ \psi \circ \delta_t)|_{\tHM[k]} = \id + \textstyle\sum_{j<k}t^{k-j}\epsilon_j.
\]
  This extends smoothly to $\id$ at $t=0$.  This completes the proof.
\end{proof}

Proposition  \ref{prop:bundle} shows that $\XH$ is isomorphic to the module of sections of a vector bundle that is isomorphic to $\tHM\times\RR$ (or $TM\times\RR$).
The isomorphism $\Phi^\psi$ is not canonical, but depends upon the choice of splitting.  
To proceed we need an alternative and {\em canonical} description of the vector bundle determined by the module $\XH$.

First, recalling the definition of $\XH\subset \Gamma^\infty(TM\times \RR)$, it is clear that for every $t\neq0$ there is a canonical restriction map
\begin{equation}
 \label{eq:evt}
 \ev_t : \XH \to \Gamma^\infty(TM); \qquad \XX\mapsto\XX|_t.
\end{equation}
In other words, the vector bundle determined by $\XH$, when restricted to $M\times \RR^\times$,
is {\em canonically} isomorphic to $TM\times \RR^\times$.

On the other hand, Proposition  \ref{prop:bundle} shows that the isomorphism $\Phi^\psi$ is independent of the choice of splitting $\psi$ at $t=0$.
Therefore, when restricted to $M\times \{0\}$,
the vector bundle determined by $\XH$ is {\em canonically} isomorphic to $\tHM$.
A splitting is used to construct the  isomorphism, but the isomorphism is independent of the splitting.

To make the isomorphism at $t=0$ more explicit, one can  define an  "evaluation at zero" map
\begin{equation}
\label{eq:ev0}
 \ev^H_0 : \XH \to \Gamma^\infty(\tHM) = \bigoplus_{i=1}^N \Gamma^\infty(H^i/H^{i-1}); \qquad 
\end{equation}
with components
\begin{align*}
\label{eq:ev0i}
 \ev_0^{(i)} : \XH &\to \Gamma^\infty(H^i/H^{i-1}) \\
 \XX &\mapsto  \sigma_i \left( \textfrac{1}{i!}\,\partial_t^i \XX|_{t=0}\right). \nonumber
\end{align*}
We leave it to the reader to check that if $\XX\in \XH$
with $\XX=\Phi^\psi(\YY)$ for 
$\YY\in\Gamma^\infty(\tHM\times\RR)$, then $\ev^H_0(\XX) = \YY|_0$. 
In other words, the evaluation map $\ev^H_0$ agrees with the restriction to $t=0$ derived from Proposition  \ref{prop:bundle}.

We now identify the vector bundle determined by the $C^\infty(M\times \RR)$-module $\XH$ as the disjoint union
\[
  \ttHM := TM\times\RR^\times \;\sqcup\; \tHM\times\{0\}
\]
The smooth structure of the total space of $\ttHM$ is implicit in the canonical identification
$\XH\cong \Gamma^\infty(\ttHM)$ (from the evaluation maps $\ev_t$ and $\ev_0^H$) and the isomorphism $\Phi^\psi$ of Proposition \ref{prop:bundle}.

To make the smooth structure explicit, any choice  of  splitting $\psi:\tHM\to TM$
gives a smooth bundle isomorphism
\begin{align}
\label{eq:bundle_iso}
 \Phi^\psi \colon \tHM\times\RR &\to \ttHM \\
 (x,v,t) &\mapsto  (x,\psi\circ\delta_t(v),t), &&t\neq0 \nonumber \\
 (x,v,0) & \mapsto (x,v,0), && t=0. \nonumber
\end{align}
Here, and in what follows, we are reappropriating the notation $\Phi^\psi$ for the vector bundle isomorphism \eqref{eq:bundle_iso}.  
This is justified by the fact that $\XH$ is  {\em canonically} isomorphic with  $\Gamma^\infty(\ttHM)$.
The map $\Phi^\psi$ in Proposition \ref{prop:bundle} is then the isomorphism $\Gamma^\infty(\tHM\times \RR) \to \Gamma^\infty(\ttHM)$  induced by  \eqref{eq:bundle_iso}.

\begin{example}
 Let $X\in\Gamma^\infty(H^m)$ be a vector field on $M$ of order $m$.  Define $\XX\in\Gamma^\infty(TM\times\RR)$ by
 \[
  \XX_t := t^m X, \qquad (t\in\RR).
 \]
 One easily confirms that $\XX\in\XH$ (see Definition \ref{def:module_of_sections}), so that $\XX$ defines a smooth section of the Lie algebroid $\ttHM$.  Its restriction to each $\ttHM|_t$ is 
 \begin{equation*}
 \begin{array}{lll}
  \ev_t(\XX) = t^mX  &\in \Gamma^\infty(TM), \hspace{2cm} & t\neq 0, \\
  \ev^H_0(\XX) = \sigma_m(X) &\in \Gamma^\infty(\tHM), & t=0.
 \end{array}
 \end{equation*}
 This example illustrates the gluing of a vector field $X$ (interpreted as a differential operator on $M$) to its principal symbol in a generalized Heisenberg calculus, via the tangent Lie algebroid $\ttHM$. This  can be easily generalized to arbitrary differential operators---see \cite[Section 11]{VanYun:PsiDOs}.
\end{example}

\section{The tangent Lie algebroid III: Lie algebroid structure}
\label{sec:III}

The next task is to understand the Lie algebroid structure on $\ttHM$.
To begin with, consider $TM \times \RR$, which is the Lie algebroid of a family of pair groupoids $M\times M \times\RR$ indexed by $t\in\RR$.  As such, it is equipped with the following Lie algebroid operations: for $\XX, \YY \in \Gamma^\infty(TM\times\RR)$,
\begin{align}
 [\XX,\YY](x,t) &= [\XX_t,\YY_t](x),   & (x,t)\in M\times\RR,  \label{eq:TMxR_bracket}\\
 \anchor(\XX) &= \XX, \label{eq:TMxR_anchor}
\end{align}
where the bracket on the right-hand side of \eqref{eq:TMxR_bracket} is the usual bracket of vector fields on $M$.

\begin{proposition}\label{prop:7a}
  The operations \eqref{eq:TMxR_bracket} and \eqref{eq:TMxR_anchor} on $\Gamma^\infty(TM\times\RR)$ restrict to the submodule $\XH$ of Definition \ref{def:module_of_sections}.
\end{proposition}

\begin{proof}
  The anchor clearly restricts to $\XH$.  
  If $\XX,\YY\in\XH$, we get
  \[
    \partial_t^k [\XX,\YY]|_{t=0}
     = \sum_{l=0}^k \left[\partial_t^l \XX|_{t=0} \; , \;\partial_t^{k-l} \XX|_{t=0}\right].
  \]
  The summand on the right lies in $[\Gamma^\infty(H^l),\Gamma^\infty(H^{k-l})] \subseteq \Gamma^\infty(H^k)$, so $[\XX,\YY]\in\XH$.
  These operations clearly satisfy the Lie algebroid axioms, since they are the restriction of the Lie algebroid operations on $\Gamma^\infty(TM\times\RR)$.
\end{proof}

Via the canonical identification $\XH\cong \Gamma^\infty(\ttHM)$ the bracket and anchor of $\XH\subset \Gamma^\infty(TM\times \RR)$ implicitly determine a Lie algebroid structure for the vector bundle $\ttHM$. 
The next proposition explicitly identifies this Lie algebroid structure.

\begin{proposition}\label{prop:7b}
 The maps $\ev_t : \XH\to \Gamma^\infty(TM)$ and  $\ev^H_0 : \XH\to \Gamma^\infty(\tHM)$ of Equations \eqref{eq:evt} and \eqref{eq:ev0} are compatible with the Lie algebroid operations. 
\end{proposition}

\begin{proof}
 Away from $t=0$ we have $\Gamma^\infty(\ttHM|_{\RR^\times}) \cong \Gamma^\infty(TM\times\RR^\times)$, so  the result for $\ev_t$ is immediate.

 At $t=0$, we calculate, for $\XX,\YY\in\XH$, 
\begin{align*}
 \ev^H_0 [\XX,\YY] &= \bigoplus_{i=1}^N \sigma_i \left( \textfrac{1}{i!} \,\partial_t^i [\XX,\YY]|_{t=0}\right) 
  = \bigoplus_{i=1}^N \sigma_i\left(\frac{1}{i!}  \sum_{j=0}^i  \binom{i}{j}[\partial_t^j \XX, \partial_t^{i-j}\YY] \right) \\
  &= \bigoplus_{i=1}^N \sum_{j=0}^i \left[\sigma_j \left(\textfrac{1}{j!}\,\partial_t^j \XX\right), \sigma_{i-j}\left(\textfrac{1}{(i-j)!}\,\partial_t^{i-j}\YY \right)\right]  
  = [\ev_0^H\XX, \ev_0^H\YY].
 \end{align*}
 The result follows.
\end{proof}

In summary, Propositions \ref{prop:7a} and \ref{prop:7b} show that the disjoint union
\begin{equation}
\label{eq:ttHM}
 \ttHM = TM\times\RR^\times \;\sqcup\; \tHM\times\{0\}
\end{equation}
is a Lie algebroid, if this union is given the smooth vector bundle structure from the previous section, as in equation \eqref{eq:bundle_iso}.


\section{The tangent groupoid}
\label{sec:tangent_groupoid}

Finally, we confirm that the tangent groupoid 
\begin{equation}
 \label{eq:TTHM}
 \TTHM = (M\times M)\times\RR^\times \;\sqcup\; \THM\times\{0\}
\end{equation}
admits a smooth structure with $\ttHM$ as its Lie algebroid.  This follows more or less directly from a result of Nistor  \cite{Nistor:integration}\footnote{Note the correction added in \cite{BenNis}.} which, roughly speaking, allows one to smoothly glue Lie groupoids together if one has a smooth glueing of their Lie algebroids.
The existence of the "integrated" Lie groupoid $\TTHM$ with Lie algebroid $\ttHM$ depends on a mild technical condition discussed in \cite{BenNis},
and in this section we verify that that condition is satisfied by $\ttHM$.

The  Lie algebroid $\ttHM$ is the disjoint union of  $\tHM\times\{0\}$ and $TM\times\RR^\times$.  Each of these components is the Lie algebra of a Lie groupoid, namely the osculating groupoid $\THM$ and the family of pair groupoids $(M\times M) \times\RR^\times$, respectively.  
As observed in \cite[Theorem 10]{BenNis}, the  condition for gluing these Lie groupoids is the local injectivity of the exponential map on $\ttHM$.  We therefore need to study the exponential maps of $\tHM$ and $M\times M$.

Defining an exponential map for a groupoid requires the choice of a leafwise connection on the Lie algebra.  We refer \cite{Nistor:integration} for the details.  We merely point out that in our case, it suffices to have a smooth family of connections on the Lie algebroids $\ttHM|_t$ for every $t\in\RR$.  

Fix a connection $\nabla$ on $\tHM$.  We will demand that $\nabla$ be graded, meaning that it is a direct sum of connections on the subbundles $\tHM[i] = H^i/H^{i-1}$.   
Equivalently, we require that
\[ \delta_t \circ \nabla \circ \delta_t^{-1} = \nabla\]
We also choose a splitting $\psi:\tHM\to TM$  as in Definition \ref{def:splitting}.  Then we can use the bundle isomorphism $\Phi^\psi:\tHM\times\RR \cong \ttHM$ of Equation \eqref{eq:bundle_iso} to push the connection $\nabla$ forward to a connection $\nabla^t$ on each of the Lie algebroids $\ttHM|_t$ with $t\in\RR$.  Specifically, because of the way we have chosen $\nabla$, we obtain $\nabla^0=\nabla$, while for each $t\neq0$, 
\begin{equation*}
 \nabla^t = (\psi\circ\delta_t) \circ \nabla \circ (\delta_t^{-1}\circ\psi^{-1}) = \psi \circ \nabla \circ \psi^{-1}.
\end{equation*}

The family of connections $\bbnabla \defeq (\nabla^t)_{t\in\RR}$ is then a smooth family of connections on the Lie algebroids $\ttHM|_t$.  As such, it induces a groupoid exponential $\Exp^{\bbnabla}$.  
In order to obtain a global smooth structure on $\TTHM$ we must verify  that  $\Exp^{\bbnabla}$ is locally injective in a neighborhood of the zero section.

Recall that the connection $\nabla^\psi \defeq \psi\circ\nabla\circ\psi^{-1}$ on $TM$ induces a groupoid exponential for the pair groupoid:
\begin{align}
\label{eq:pair_exponential}
 \Exp^{\nabla^\psi} : TM  & \to M \times M \\
   (x,v) &\mapsto (\exp^{\nabla^\psi}_x(v), x), \nonumber
\end{align}
where $\exp^{\nabla^\psi}$ is the usual geometric exponential.
Strictly speaking, $\Exp^{\nabla^\psi}$ may be only defined on some open neighbourhood of the zero section in $TM$.  The next definition deals with this eventuality.  Note that it will be convenient to pull back everything to the associated graded bundle $\tHM$ via the splitting $\psi:\tHM \to TM$.

\begin{definition}
 \label{def:domain_of_injectivity}
An open neighbourhood of the zero section $U \subseteq \tHM$ will be called a {\em domain of injectivity} (with respect to the connection $\nabla$ and splitting $\psi$) if the map  $\Exp^{\nabla^\psi}\!\circ\,\psi:  U \to M\times M$ is well-defined and injective.  
\end{definition}

We can now describe the exponential map of the full tangent groupoid $\TTHM$.  The explicit formula is most conveniently expressed by using the bundle isomorphism $\Phi^\psi : \tHM\times\RR \cong \ttHM$ of Equation \eqref{eq:bundle_iso} to pull back from the Lie algebroid $\ttHM$ to $\tHM\times\RR$.

\begin{lemma}
\label{lem:EExp}
 Let $U\subseteq\tHM$ be a domain of injectivity as above.  The composition $\Exp^{\bbnabla} \!\circ \,\Phi^\psi: \tHM\times\RR \to \TTHM$ is well-defined and injective on the subset
 \begin{equation}
 \label{eq:exponential domain}
  \UU\defeq \{ (x,v,t) \in \tHM\times\RR \st (x,\psi(\delta_tv)) \in U \} \subseteq \tHM\times\RR,
 \end{equation}
 with image
 \begin{equation}
 \label{eq:exponential_patch}
 (\THM\times\{0\}) \sqcup (\Exp^{\nabla^\psi}(U) \times \RR^\times) \subset \TTHM .
\end{equation}
 It is given explicitly by the formula
 \begin{equation}
 \label{eq:EExp}
  \Exp^{\bbnabla}\!\circ\,\Phi^\psi : (x,v,t) \mapsto 
  \begin{cases}
    (\Exp^{\nabla^\psi}(\psi\circ\delta_t(v)),t), & t\neq 0 \\
   (x,v,0), & t=0.
  \end{cases}
 \end{equation}
\end{lemma}

\begin{proof}
 On the fibre $\tHM\times\{0\}$, both $\Exp^{\bbnabla}$ and $\Phi^\psi$ are the identity.  For the fibres at $t\neq0$, the stated formula follows by composing Equation \eqref{eq:bundle_iso} for $\Phi^\psi$ with Equation \eqref{eq:pair_exponential} for the exponential of the pair groupoid.  The result follows.
 \end{proof}

Note that the subset $\UU \subset \tHM \times \RR$ of the lemma is an open neighbourhood both of the zero section $M\times\{0\}\times\RR$ and of the fibre $\tHM\times\{0\}$ at $t=0$.  We will refer to the maps of the form \eqref{eq:EExp} as \emph{global exponential charts}.

\begin{theorem}
\label{thm:smooth_structure}
The groupoid $\TTHM$ admits a unique smooth structure such that $\ttHM$ is its Lie algebra, and such that for any splitting $\psi:\tHM\to TM$, any graded connection $\nabla$ on $\tHM$ and any domain of injectivity $U\subseteq \tHM$, the global exponential chart $\Exp^{\bbnabla}\!\circ\,\Phi^\psi$ is a smooth diffeomorphism from $\UU$ (as in Lemma \ref{lem:EExp}) to its image.
\end{theorem}

\begin{proof}
By Lemma \ref{lem:EExp}, the groupoid $\TTHM$ satisfies the injectivity condition of Theorem 10 in  \cite{BenNis}. As stated, that theorem applies only to $s$-simply connected groupoids, i.e.\ it will provide us with a smooth structure on the groupoid
\[
 \widetilde{\TTHM} \defeq \Pi_M \times \RR^\times \; \sqcup \; \tHM \times \{0\}
\]
where $\Pi_M$ is the fundamental groupoid of $M$.  But this smooth structure descends also to $\TTHM$.  We simply take an atlas for $\TTHM$ comprising the standard atlas for $M\times M \times \RR^\times$ together with the global exponential charts of Lemma \ref{lem:EExp}.  Since the global exponential charts yield diffeomorphic neighbourhoods in $\TTHM$ and $\widetilde{\TTHM}$, the result follows.
\end{proof}


\section{Adiabatic groupoids}
\label{sec:filtered_groupoids}

Connes' tangent groupoid is a particular example of the more general construction of an adiabatic groupoid, which is a deformation of a Lie groupoid $G$ to its Lie algebroid $\AG$.  This generalization becomes important when one deals with pseudodifferential operators on singular spaces, such as foliations, manifolds with boundaries, Lie manifolds and stratified manifolds.  A complete bibliography here would take too much space, but for a sampling of the groupoids involved, one could start with \cite{Connes:Integration,  MonPie, NisWeiXu,DebLesRoc}.  Following the references in these papers will give a much more rounded bibliography.

There is nothing preventing us from applying our filtered construction at this level of generality, which would permit the unification of singular manifolds with filtrations on the tangent space.
In this section, we outline the construction of the adiabatic groupoid of a filtered groupoid.  The proofs are straightforward generalizations of those above, and will be omitted.

\subsection{Filtered groupoids and the osculating groupoid}

For simplicity, we shall work here with Hausdorff Lie groupoids, although the construction works equally well for almost differentiable \cite{NisWeiXu} or longitudinally smooth \cite{Monthubert:groupoids} groupoids---\emph{i.e.}, groupoids in which the fibres are manifolds while the base is a manifold with corners.

\begin{definition}
 A Lie groupoid $G$ is called \emph{filtered} if its Lie algebroid $\AG$ is equipped with a filtration by subbundles
 \[
  0 = \mathrm{A}^0 G \leq \mathrm{A}^1 G \leq \cdots \leq \mathrm{A}^N G = \AG
 \]
 such that the module of sections $\Gamma^\infty(\mathrm{A}^\bullet G)$ is a filtered Lie algebra.
\end{definition}

The \emph{osculating Lie algebroid} $\aHG$ is the associated graded bundle of $ \AG$:
\[
 \aHG := \bigoplus \aHG[i], \qquad \text{where } \aHG[i] = \mathrm{A}^iG / \mathrm{A}^{i-1}G.
\]
It is a Lie algebroid over $M$ with the bracket of sections induced from that on $\mathrm{A}^\bullet G$ and the zero anchor.  As such, it is a bundle of nilpotent Lie algebras.  

We denote the grading maps by $\sigma_i : \mathrm{A}^iG\to \mathrm{A}^iG/\mathrm{A}^{i-1}G \subseteq \aHG$.  A bundle isomorphism $\psi: \AG \to \aHG$ is called a \emph{splitting} if $\psi|_{\mathrm{A}^iG}$ is a splitting of $\sigma_i$ for all $i$.

There is again a one-parameter family $(\delta_\lambda)_{\lambda\in\RR^\times}$ of \emph{dilations} on $\aHG$, where $\delta_\lambda$ acts on $\aHG[i]$ by multiplication by $\lambda^i$.  These are Lie algebra automorphisms.

The \emph{osculating groupoid} $\AHG$ is the bundle of connected, simply connected nilpotent Lie groups which integrates $\aHG$. The dilations $(\delta_\lambda)_{\lambda\in\RR^\times}$ integrate to a one-parameter family of groupoid automorphisms of $\AHG$.

\subsection{The adiabatic Lie algebroid}

As before, we define the filtered adiabatic groupoid by starting with the module of sections of its Lie algebroid:
\begin{equation*}
\label{eq:sections}
   \Gamma^\infty(\aaHG) := \{ \XX \in \Gamma^\infty( \AG\times\RR) \st \partial_t^k\XX|_{t=0} \in \Gamma^\infty(\mathrm{A}^kG) \text{ for all } k\geq0\}.
\end{equation*}
This $C^\infty(M\times\RR)$-module inherits a Lie bracket and anchor by restricting those of $\Gamma^\infty( \AG\times\RR)$.  Upon fixing a splitting, $\Gamma^\infty(\aaHG)$ is isomorphic to $\Gamma^\infty( \aHG\times\RR)$ via the analogue of the map $\Phi^\psi$ from Equation \ref{prop:bundle}.  We thereby obtain a smooth bundle $\aaHG$ over $M\times\RR$ whose module of sections corresponds with $\Gamma^\infty(\aaHG)$.  The smooth structure on $\aaHG$ is independent of the choice of splitting.

The maps
 \begin{align*}
  \ev_t : &\Gamma^\infty(\aaHG) \to \Gamma^\infty( \AG); \qquad &\XX \to &\XX|_t,\\
  \ev_0^H : &\Gamma^\infty(\aaHG) \to \Gamma^\infty(\aHG); &\XX \to &
    \bigoplus_i \sigma_i \left( \textfrac{1}{i!}\,\partial_t^i \XX|_{t=0}\right)
 \end{align*}
 are compatible with Lie brackets and anchors so that, algebraically,
 \[
  \aaHG \cong ( \AG \times \RR^\times_+) \sqcup (\aHG\times\{0\}).
 \]

\subsection{The adiabatic groupoid}

The \emph{adiabatic groupoid} of a filtered groupoid is
\[
  \AAHG = ( G \times \RR^\times_+) \sqcup (\AHG \times\{0\}).
\]
Each of the two components has an obvious smooth structure.  To obtain the global smooth structure, we need exponential maps. 

For this, we fix a graded connection $\nabla$ on $\aHG$ and a splitting $\psi:\aHG \to  \AG$.  We write $\nabla^\psi = \psi\circ\nabla\circ\psi^{-1}$, for the induced connection on $ \AG$, and $\Exp^{\nabla^\psi}: \AG \supset U \to G$ for the  associated groupoid exponential, where $U$ is a domain of injectivity for $\Exp^{\nabla^\psi}$.  
Then 
we have a well-defined bijective map from the neighbourhood
\begin{equation}
\label{eq:adiabatic_global_coords}
 \UU \defeq \{(x,\xi,t) \in \aHG \times\RR \st (x,\psi\circ\delta_t(\xi)) \in U\} \subset \aHG \times\RR
\end{equation}
to the set
\begin{equation}
 \left( \aHG \times \{0\} \right) \sqcup \left( \Exp^{\nabla^\psi}(U) \times \RR^\times \right) \subset \AAHG
\end{equation}
given by
\begin{equation}
\label{eq:adiabatic_EExp}
 \EExp^{\bbnabla}\!\circ\,\Phi^\psi :
  (x,\xi,t) \mapsto 
  \begin{cases}
  (\Exp^{\nabla^\psi}_x\psi(\delta_t\xi),t), & t\neq 0, \\
 (x,\xi,0), & t=0,
 \end{cases}
\end{equation}
where $\Phi^\psi: \aHG\times\RR \to \aaHG$ is the obvious analogue of the bundle isomorphism \eqref{eq:bundle_iso}.   Declaring this to be a diffeomorphism determines a smooth structure on $\AAHG$.

\begin{remark}
 This construction leads one naturally to imagine questions of analysis on manifolds with boundaries or corners in the presence of a Lie filtration.  Compare, \emph{e.g.}, \cite{Monthubert:groupoids}, \cite{AmmLauNis}, \cite{DebLesRoc} for the unfiltered case.  We hope to return to this in a future work.
\end{remark}


\bibliographystyle{alpha}
\bibliography{refs}

\def\polhk#1{\setbox0=\hbox{#1}{\ooalign{\hidewidth
  \lower1.5ex\hbox{`}\hidewidth\crcr\unhbox0}}}
\begin{thebibliography}{CGGP92}

\bibitem[ALN07]{AmmLauNis}
Bernd Ammann, Robert Lauter, and Victor Nistor.
\newblock Pseudodifferential operators on manifolds with a {L}ie structure at
  infinity.
\newblock {\em Ann. of Math. (2)}, 165(3):717--747, 2007.

\bibitem[BN03]{BenNis}
Moulay-Tahar Benameur and Victor Nistor.
\newblock Homology of algebras of families of pseudodifferential operators.
\newblock {\em J. Funct. Anal.}, 205(1):1--36, 2003.

\bibitem[CGGP92]{ChrGelGloPol}
Michael Christ, Daryl Geller, Pawe{\l} G{\l}owacki, and Larry Polin.
\newblock Pseudodifferential operators on groups with dilations.
\newblock {\em Duke Math. J.}, 68(1):31--65, 1992.

\bibitem[Con79]{Connes:Integration}
Alain Connes.
\newblock Sur la th\'eorie non commutative de l'int\'egration.
\newblock In {\em Alg\`ebres d'op\'erateurs ({S}\'em., {L}es {P}lans-sur-{B}ex,
  1978)}, volume 725 of {\em Lecture Notes in Math.}, pages 19--143. Springer,
  Berlin, 1979.

\bibitem[Con94]{Connes:NCG}
Alain Connes.
\newblock {\em Noncommutative geometry}.
\newblock Academic Press, Inc., San Diego, CA, 1994.

\bibitem[CP15]{ChoPon}
Woocheol Choi and Raphael Ponge.
\newblock Privileged coordinates and tangent groupoid for carnot manifolds.
\newblock Preprint. \url{http://arxiv.org/abs/1510.05851}, 2015.

\bibitem[DLR15]{DebLesRoc}
Claire Debord, Jean-Marie Lescure, and Fr\'ed\'eric Rochon.
\newblock Pseudodifferential operators on manifolds with fibred corners.
\newblock {\em Ann. Inst. Fourier (Grenoble)}, 65(4):1799--1880, 2015.

\bibitem[FS74]{FolSte:Estimates}
G.~B. Folland and E.~M. Stein.
\newblock Estimates for the {$\bar \partial _{b}$} complex and analysis on the
  {H}eisenberg group.
\newblock {\em Comm. Pure Appl. Math.}, 27:429--522, 1974.

\bibitem[HN79]{HelNou}
B.~Helffer and J.~Nourrigat.
\newblock Caracterisation des op\'erateurs hypoelliptiques homog\`enes
  invariants \`a gauche sur un groupe de {L}ie nilpotent gradu\'e.
\newblock {\em Comm. Partial Differential Equations}, 4(8):899--958, 1979.

\bibitem[Kna02]{Knapp:Lie_groups}
Anthony~W. Knapp.
\newblock {\em Lie groups beyond an introduction}, volume 140 of {\em Progress
  in Mathematics}.
\newblock Birkh\"auser Boston, Inc., Boston, MA, second edition, 2002.

\bibitem[Mel82]{Melin:Lie_filtrations}
Anders Melin.
\newblock Lie filtrations and pseudo-differential operators.
\newblock Preprint, 1982.

\bibitem[MM03]{MoeMrc}
I.~Moerdijk and J.~Mr{\v{c}}un.
\newblock {\em Introduction to foliations and {L}ie groupoids}, volume~91 of
  {\em Cambridge Studies in Advanced Mathematics}.
\newblock Cambridge University Press, Cambridge, 2003.

\bibitem[Mon99]{Monthubert:groupoids}
Bertrand Monthubert.
\newblock Pseudodifferential calculus on manifolds with corners and groupoids.
\newblock {\em Proc. Amer. Math. Soc.}, 127(10):2871--2881, 1999.

\bibitem[MP97]{MonPie}
Bertrand Monthubert and Fran\c{c}ois Pierrot.
\newblock Indice analytique et groupo\"\i des de {L}ie.
\newblock {\em C. R. Acad. Sci. Paris S\'er. I Math.}, 325(2):193--198, 1997.

\bibitem[Nis00]{Nistor:integration}
Victor Nistor.
\newblock Groupoids and the integration of {L}ie algebroids.
\newblock {\em J. Math. Soc. Japan}, 52(4):847--868, 2000.

\bibitem[NWX99]{NisWeiXu}
Victor Nistor, Alan Weinstein, and Ping Xu.
\newblock Pseudodifferential operators on differential groupoids.
\newblock {\em Pacific J. Math.}, 189(1):117--152, 1999.

\bibitem[Pon06]{Ponge:groupoid}
Rapha{\"e}l Ponge.
\newblock The tangent groupoid of a {H}eisenberg manifold.
\newblock {\em Pacific J. Math.}, 227(1):151--175, 2006.

\bibitem[SH16]{HigSad}
Ahmad Reza Haj~Saeedi Sadegh and Nigel Higson.
\newblock Euler-like vector fields, deformation spaces and manifolds with
  filtered structure.
\newblock Preprint. \url{http://arxiv.org/abs/1611.05312}, 2016.

\bibitem[Tay]{Taylor:Microlocal}
Michael Taylor.
\newblock Noncommutative microlocal analysis, part {I} (revised version).
\newblock http://www.unc.edu/math/Faculty/met/NCMLMS.pdf.

\bibitem[vE05]{VanErp:Thesis}
Erik van Erp.
\newblock {\em The {A}tiyah-{S}inger index formula for subelliptic operators on
  contact manifolds}.
\newblock ProQuest LLC, Ann Arbor, MI, 2005.
\newblock Thesis (Ph.D.)--The Pennsylvania State University.

\bibitem[vE10a]{VanErp:I}
Erik van Erp.
\newblock The {A}tiyah-{S}inger index formula for subelliptic operators on
  contact manifolds. {P}art {I}.
\newblock {\em Ann. of Math. (2)}, 171(3):1647--1681, 2010.

\bibitem[vE10b]{VanErp:II}
Erik van Erp.
\newblock The {A}tiyah-{S}inger index formula for subelliptic operators on
  contact manifolds. {P}art {II}.
\newblock {\em Ann. of Math. (2)}, 171(3):1683--1706, 2010.

\bibitem[vEY15]{VanYun:PsiDOs}
Erik van Erp and Robert Yuncken.
\newblock A groupoid approach to pseudodifferential operators.
\newblock Preprint. \url{http://arxiv.org/abs/1511.01041}, 2015.

\end{thebibliography}

\end{document}